\documentclass[letterpaper,11pt]{amsart}
\usepackage[margin=1in]{geometry}
\usepackage{hyperref}
\usepackage{amsfonts,amssymb}
\usepackage{pb-diagram}
\usepackage{enumerate}
\makeatletter
\let\@@enum@org\@@enum@
\def\@@enum@[#1]{\@@enum@org[\normalfont #1]}
\makeatother

\newcommand\form[1]{\langle #1\rangle}
\newcommand\supp{\operatorname{supp}}

\title{An obstruction to embedding right-angled Artin groups in mapping class groups}

\newtheorem{thm}{Theorem}
\newtheorem{lemma}[thm]{Lemma}

\newtheorem{quest}[thm]{Question}

\numberwithin{thm}{section}

\newcommand{\gam}{\Gamma}

\DeclareMathOperator{\Mod}{Mod}

\author{Sang-hyun Kim}
\address{Department of Mathematical Sciences, KAIST, 335 Gwahangno, Yuseong-gu, Daejeon 305-701, Republic of Korea}
\email{shkim@kaist.edu}
\thanks{The first named author is supported by the Basic Science Research Program (2011-0026138) and the Mid-Career Researcher Program (2011-0027600) through the National Research Foundation funded by the Ministry of Education, Science and Technology of Korea.}

\author{Thomas Koberda}
\address{Department of Mathematics\\ Yale University\\ P.O. Box 208283\\ New Haven, CT 06520-8283}
\email{thomas.koberda@gmail.com}
\keywords{}

\begin{document}
\begin{abstract}
For every orientable surface of finite negative Euler characteristic, we find a right-angled Artin group of cohomological dimension two which does not embed into the associated mapping class group.  For a right-angled Artin group on a graph $\gam$ to embed into the mapping class group of a surface $S$, we show that the chromatic number of $\gam$ cannot exceed the chromatic number of the clique graph of the curve graph $\mathcal{C}(S)$.  Thus, the chromatic number of $\gam$ is a global obstruction to embedding the right-angled Artin group $A(\gam)$ into the mapping class group $\Mod(S)$.
\end{abstract}
\maketitle
\begin{center}
\today
\end{center}
\section{Introduction}
Let $S$ be an orientable surface with finite negative Euler characteristic, and let $\Mod(S)$ denote the mapping class group of $S$.  Let $\gam$ be a finite simplicial graph. We denote the vertex set and the edge set of $\gam$ by $V(\gam)$ and $E(\gam)$, respectively.  The \emph{right-angled Artin group} $A(\gam)$ is defined by its presentation \[A(\gam)=\langle V(\gam)\mid [u,v]=1\,\, {\rm whenever } \,\, \{u,v\}\in E(\gam)\rangle.\]

It is known by the work of Crisp and Farb (cf. \cite{CF}, also \cite{CP}, \cite{CW}, \cite{clm2010} and \cite{kobraag}) that each right-angled Artin group embeds into some mapping class group.  In this article, we are interested in obstructions to embedding a given right-angled Artin group into a given mapping class group.  In particular, we are interested in the following question:
\begin{quest}
Given a right-angled Artin group $A(\gam)$, what is the simplest surface $S$ for which there is an embedding \[A(\gam)\to\Mod(S)?\]
\end{quest}
Here, by ``simplest surface", we mean one with the smallest absolute value of the Euler characteristic.
Given a right-angled Artin group $A(\gam)$, there are many restrictions on the simplest surface $S$ for which $A(\gam)$ embeds in $\Mod(S)$.  Possibly the simplest restriction comes from 
the maximal rank of an abelian subgroup.  
Indeed, Birman, Lubotzky and McCarthy (see \cite{BLM}) showed that the maximal rank of an abelian subgroup of $\Mod(S)$ is equal to the maximum number of the components in a multicurve on $S$ and is thus roughly proportional to the Euler characteristic of $S$.  
A \emph{clique} of a simplicial graph is a set of pairwise adjacent vertices.  By convention, a singleton vertex is a clique.
The maximal rank of an abelian subgroup of a right-angled Artin group $A(\gam)$ is equal to the largest size of a clique of $\gam$, which is also equal to the cohomological dimension of $A(\gam)$.
Thus, if $A(\gam)$ embeds in $\Mod(S)$ then the cohomological dimension of $A(\gam)$ cannot exceed the maximal rank of an abelian subgroup of $\Mod(S)$.
If $P=F_1\times\cdots\times F_k$ is a product of finitely generated, nonabelian free groups, one can formulate similar restrictions on a surface $S$ for which $P$ embeds into $\Mod(S)$ as for abelian subgroups.  A discussion of such restrictions can be found in \cite{kobraag}.

The restrictions on $S$ which arise from abelian subgroups and products of free groups are of the same general flavor, coming from the fact that to accommodate more direct factors, one needs more ``disjoint subsurfaces".  The purpose of this note is to give new restrictions on $S$ which can be attributed more to the global structure of the graph $\gam$ rather than to local structure (such as the largest size of a clique in $\gam$).

Our obstruction comes from comparing the \emph{chromatic numbers} of $\gam$ and of the \emph{curve graph} of $S$.  Recall that the chromatic number of a graph $\gam$ is the smallest number of colors needed to color the vertices of $\gam$ in such a way that no two adjacent vertices have the same color.  The curve graph of a surface $S$ is a graph whose vertices are isotopy classes of essential, nonperipheral, simple closed curves in $S$.  Two vertices are connected by an edge if the two isotopy classes of curves can be disjointly realized.  The \emph{curve complex} of $S$ is the associated flag complex of the curve graph.  Our main result is the following:

\begin{thm}\label{t:embed}
Let $S$ be an orientable surface of finite negative Euler characteristic, and let $M$ be a positive integer.  Then there is a finite graph $\gam_{S,M}$ of girth at least $M$ such that $A(\gam_{S,M})$ does not embed as a subgroup of $\Mod(S)$.
\end{thm}

Recall that the \emph{girth} of a graph is the shortest cycle length.  Thus, a graph of girth at least four has no triangles and the associated right-angled Artin group has cohomological dimension two.

\section{Acknowledgements}
The authors wish to thank B. Farb and C. Leininger for helpful conversations concerning the chromatic number of the curve graph. The authors wish to thank the Park City Mathematical Institute summer session in geometric group theory for hospitality while this research was completed. 
The authors are grateful for the suggested corrections from an anonymous referee.

\section{The proof of Theorem \ref{t:embed}}
In order to prove Theorem \ref{t:embed}, we will need the following result, whose proof is elementary:

\begin{thm}[Bestvina--Bromberg--Fujiwara, \cite{bestbromfuji}, Lemma 4.6]\label{p:chromatic}
Let $S$ be an orientable surface of finite negative Euler characteristic.  Then the curve graph has a finite chromatic number.
\end{thm}

Throughout, write $\gam$ for a finite simplicial graph and let $\Mod(S)$ denote the mapping class group of a surface $S$ as above.  

Let $Y$ be a (possibly infinite) simplicial graph.  
We will write $Y_k$ for the \emph{clique graph} of $Y$, which is defined as follows.
The vertex set of $Y_k$ consists of a vertex $v_K$ for every clique $K$ of $Y$. Two vertices $v_K$ and $v_L$ are adjacent in $Y_k$ if and only if $K\cup L$ is also a clique in $Y$.

For a set $F$, a map $f: Y^{(0)}\to F$ is called a \emph{coloring} of $Y$ by $F$ if whenever $u$ and $v$ are adjacent vertices, we have that $f(u)\ne f(v)$.  Note that the minimal $|F|$ for which there exists a coloring $f:Y^{(0)}\to F$ is the chromatic number of $Y$.

\begin{lemma}\label{l:cliquecolor}
If $Y$ is a simplicial graph with a finite chromatic number then $Y_k$ has a finite chromatic number.
\end{lemma}

\begin{proof}
Let $f:Y^{(0)}\to F$ be a coloring of $Y$ by a finite set $F$.
We define $g:Y_k^{(0)}\to 2^F$ by $g(v_K)=f(K)$ for each clique $K$ of $Y$.
We claim that $g$ is a coloring of $Y_k$ by the finite set $2^F$.
Suppose $K$ and $L$ are distinct cliques in $Y$ such that $v_K$ and $v_L$ are adjacent.
There exists $v\in (K\setminus L)\cup (L\setminus K)$. We may assume $v\in K\setminus L$.
Since $K\cup L$ is a clique, the vertex $v$ is adjacent to each vertex in $L$.
In particular, $f(v)\not\in f(L)$. We have  $f(v)\in g(v_K)=f(K)\ne f(L)=g(v_L)$.
\end{proof}

We see from the above proof that if $Y$ is a simplicial graph with chromatic number $M$ then $Y_k$ has chromatic number at most $2^M$.  
We say a graph $X$ is an \emph{induced subgraph} of $Y$ if $X$ is a subgraph of $Y$ and every edge of $Y$ joining two vertices in $X$ is contained in $X$.
For an element $g$ of $A(\gam)$, the \emph{support} of $g$ is $\supp(g) = \{v\in V(\gam)\mid v \text{ appears in a reduced word representing }g\}$.  It is true, though not completely obvious, that the support of an element $g\in A(\gam)$ is well--defined.

\begin{lemma}\label{l:embed}
Suppose that $A(\gam)$ embeds as a subgroup of $\Mod(S)$.  Then $\gam$ embeds as an induced subgraph of $\mathcal{C}(S)_k$.
\end{lemma}
\begin{proof}
Choose an embedding $\phi:A(\gam)\to\Mod(S)$.  
By replacing $\phi(v)$ by some positive power, we may assume that for each $v$ the mapping class $\phi(v)$ is \emph{pure}; that is, each $\phi(v)$ can be written as a product of commuting mapping classes which are either powers of Dehn twists about simple closed curves or pseudo-Anosov mapping classes on connected, incompressible subsurfaces of $S$.  
Decompose $\phi(v)$ in such a way, writing 
\[\phi(v)=f_1^v\cdots f_{n(v)}^v.\]  
There exists a collection of mapping classes $C$ such that $f_i^v$ is a power of some element in $C$ for every vertex $v$ of $\gam$ and $1\le i\le n(v)$.
By choosing $C$ with the smallest cardinality, we may assume that no two elements of $C$ generate a cyclic subgroup of $\Mod(S)$.
Let $X$ be the \emph{commutation graph} of $C$; that is, the vertices of $X$ are the elements of $C$ and two vertices are adjacent if the associated mapping classes commute.
The main result of~\cite{kobraag} implies that $\form{C}\le\Mod(S)$ is isomorphic to $A(X)$, again possibly after replacing each $\phi(v)$ by a further higher power and accordingly, raising each element of $C$ to a suitable power.
We will argue that $X$ embeds in $\mathcal{C}(S)$, and then that $\gam$ embeds in $X_k$, thus establishing the lemma.

Write $C$ as a union of powers of Dehn twists  $\{g_1,\ldots,g_p\}$ about simple closed curves $\{\alpha_1,\ldots,\alpha_p\}$ and of pseudo-Anosov mapping classes $\{g_{p+1},\ldots,g_{p+q}\}$ supported on subsurfaces $\{S_{p+1},\ldots,S_{p+q}\}$.  
We let $S_i$ be a sufficiently small regular neighbourhood of $\alpha_i$ for $i\le p$. 
The mapping classes $g_i$ and $g_j$ correspond to adjacent vertices of $X$ if and only if $S_i$ and $S_j$ are disjoint.  
For each $i=p+1, p+2,\ldots,p+q$, we inductively choose an essential, nonperipheral simple closed curve $\beta_i$ in the interior of $S_i$ and apply a sufficiently large power of $g_i$ to $\beta_i$ to get $\alpha_i$ so that the following holds: 
\begin{enumerate}[(i)]
\item $\alpha_i$ and $\alpha_j$ are disjoint if and only if so are $S_i$ and $\alpha_j$, for $j=1,2,\ldots,i-1$;
\item $\alpha_i$ and $S_j$ are disjoint if and only if so are $S_i$ and $S_j$, for $j=i+1,i+2,\ldots,p+q$.
\end{enumerate}
After this inductive process, we have that $\alpha_i$ and $\alpha_j$ are disjoint if and only if $S_i$ and $S_j$ are disjoint for  $1\le i, j\le p+q$. Furthermore, we can require the curves $\alpha_i$ and $\alpha_j$ to be non-isotopic for $i\ne j$; note that $S_i$ and $S_j$ may be isotopic.
Hence $X$ coincides with the subgraph of  $\mathcal{C}(S)$ induced by $Y=\{\alpha_1, \ldots,\alpha_{p+q}\}$.

Let $\psi$ be an embedding from $A(\gam)$ into $A(X)$ such that $\supp(\psi(v))$ is a clique for each vertex $v$ of $\gam$ and moreover, $\sum_{v\in V(\gam)} |\supp(\psi(v))|$ is minimal. Such a $\psi$ exists, since the composition of $\phi:A(\gam)\to \form{C}$ with the isomorphism $\form{C}\cong A(X)$ is an embedding that sends  each vertex of $\gam$ to a word whose support is a clique. We claim that $\supp(\psi(v))\ne\supp(\psi(v'))$ for distinct vertices $v$ and $v'$ of $\gam$.
Suppose not, and write $\psi(v) = x_1^{p_1}\cdots x_k^{p_k}$ and $\psi(v') = x_1^{q_1}\cdots x_k^{q_k}$ such that 
$k>0$, \[\prod_i p_i\ne0\ne \prod_i q_i\] and $\{x_1,\ldots,x_k\}$ span a clique in $X$.
Note that for each vertex $u$ of $\gam$ adjacent to $v$, we have $[\psi(v),\psi(u)]=1$ and so, \[\supp(\psi(v'))\cup\supp(\psi(u)) =\supp(\psi(v))\cup\supp(\psi(u))\] spans a clique; this implies that $[\psi(v'),\psi(u)]=1$ and $v'$ is adjacent to $u$. Hence we have an automorphism\footnote{An automorphism of this kind is called a \emph{transvection}~\cite{Servatius1989}.} $\xi:A(\gam)\to A(\gam)$ defined by $\xi(v) = vv'^{-1}$ and $\xi(x) = x$ for each $x\in V(\gam)\setminus \{v\}$. Let $\eta:A(\gam)\to A(\gam)$ be an embedding that maps $v$ to $v^{q_1}$ while fixing all the other vertices. Then \[\psi' =  \psi \circ \eta \circ \xi^{p_1}:A(\gam)\to A(X)\] is an embedding such that \[\supp(\psi'(v))\subseteq\{x_2,\ldots,x_k\}\subsetneq \supp(\psi(v))\] 
and that $\psi'(y) = \psi(y)$ for $y\subseteq V(\gam)\setminus\{v\}$. This contradicts the minimality of supports assumption.
 
We now define a map $\delta:V(\gam)\to V(X_k)$ by $\delta(v) = \supp(\psi(v))$. This map is an embedding by the preceding paragraph. Moreover, $\delta$ is easily seen to extend to a graph embedding $\gam\to X_k$ such that the $\delta(\gam)$ is an induced subgraph of $X_k$.
\end{proof}

We remark that in the proof above, sufficiently high powers of Dehn twists about elements of $Y$ generate a subgroup of $\Mod(S)$ which is isomorphic to $A(X)$, by the principal result of \cite{kobraag}.  The last ingredient for the proof of Theorem \ref{t:embed} is the following result of Erd\"os:

\begin{thm}
Let $N$ and $M$ be positive integers.  There exists a graph $\gam_{N,M}$ of girth at least $M$ and chromatic number at least $N$.
\end{thm}
\begin{proof}
See \cite{diestel}.
\end{proof}

We can now prove our main result:

\begin{proof}[Proof of Theorem \ref{t:embed}]
Suppose that $A(\gam)<\Mod(S)$.   By Lemma \ref{l:embed}, we have that $\gam$ embeds as a subgraph of $\mathcal{C}(S)_k$.  Combining Theorem \ref{p:chromatic} and Lemma \ref{l:cliquecolor}, we see that $\mathcal{C}(S)_k$ has a finite chromatic number $N(S)$, so that the chromatic number of $\gam$ cannot exceed $N(S)$.  By Erd\"os' Theorem, for every $M$ we can produce a graph of girth at least $M$ whose chromatic number exceeds $N(S)$.  The corresponding right-angled Artin group cannot embed in $\Mod(S)$.
\end{proof}

\end{document}